\documentclass[12pt,reqno]{amsart}
\usepackage{latexsym}
\usepackage{amssymb}
\usepackage{mathrsfs}
\usepackage{amsmath}
\usepackage{fancybox,color}
\usepackage{enumerate}
\usepackage[latin1]{inputenc}
\usepackage{soul} 
\usepackage[colorlinks=true, linkcolor=blue, citecolor=blue]{hyperref}

\usepackage{color}

\usepackage[left=2.2cm,top=2.3cm,right=2.2cm]{geometry}

\def\1{\raisebox{2pt}{\rm{$\chi$}}}

\newtheorem{theorem}{Theorem}[section]

\newtheorem{lemma}[theorem]{Lemma}

\newtheorem{definition}[theorem]{Definition}

\theoremstyle{definition}
\newtheorem{remark}[theorem]{Remark}

\newcommand{\R}{{\mathbb R}}

\newcommand{\N}{{\mathbb N}}
\newcommand{\Z}{{\mathbb Z}}

\newcommand{\M}{{\mathcal M}}

\newcommand\diam{\operatorname{diam}}

\newcommand{\defeq}{\mathrel{\mathop:}=} 
\newcommand{\co}{\mskip0.5mu\colon\thinspace}   

\newcommand{\eps}{{\varepsilon}}
\def\1{\raisebox{2pt}{\rm{$\chi$}}}

\newcommand{\Lip}{\operatorname{Lip}}

\newcommand{\len}{\operatorname{len}} 
\def\vint_#1{\mathchoice%
        {\mathop{\kern 0.2em\vrule width 0.6em height 0.69678ex depth -0.58065ex
                \kern -0.8em \intop}\nolimits_{\kern -0.4em#1}}%
        {\mathop{\kern 0.1em\vrule width 0.5em height 0.69678ex depth -0.60387ex
                \kern -0.6em \intop}\nolimits_{#1}}%
        {\mathop{\kern 0.1em\vrule width 0.5em height 0.69678ex
            depth -0.60387ex
                \kern -0.6em \intop}\nolimits_{#1}}%
        {\mathop{\kern 0.1em\vrule width 0.5em height 0.69678ex depth -0.60387ex
                \kern -0.6em \intop}\nolimits_{#1}}}
\def\vintslides_#1{\mathchoice%
        {\mathop{\kern 0.1em\vrule width 0.5em height 0.697ex depth -0.581ex
                \kern -0.6em \intop}\nolimits_{\kern -0.4em#1}}%
        {\mathop{\kern 0.1em\vrule width 0.3em height 0.697ex depth -0.604ex
                \kern -0.4em \intop}\nolimits_{#1}}%
        {\mathop{\kern 0.1em\vrule width 0.3em height 0.697ex depth -0.604ex
                \kern -0.4em \intop}\nolimits_{#1}}%
        {\mathop{\kern 0.1em\vrule width 0.3em height 0.697ex depth -0.604ex
                \kern -0.4em \intop}\nolimits_{#1}}}

\newcommand{\aveint}[2]{\mathchoice%
        {\mathop{\kern 0.2em\vrule width 0.6em height 0.69678ex depth -0.58065ex
                \kern -0.8em \intop}\nolimits_{\kern -0.45em#1}^{#2}}%
        {\mathop{\kern 0.1em\vrule width 0.5em height 0.69678ex depth -0.60387ex
                \kern -0.6em \intop}\nolimits_{#1}^{#2}}%
        {\mathop{\kern 0.1em\vrule width 0.5em height 0.69678ex depth -0.60387ex
                \kern -0.6em \intop}\nolimits_{#1}^{#2}}%
        {\mathop{\kern 0.1em\vrule width 0.5em height 0.69678ex depth -0.60387ex
                \kern -0.6em \intop}\nolimits_{#1}^{#2}}}

\title[Self-improvement of the pointwise Hardy inequality]{Self-improvement of pointwise Hardy inequality}

\author[S.\! Eriksson-Bique]{Sylvester Eriksson-Bique}   
\address[S.E.-B.]{  Department of mathematics, UCLA, 520 Portola Plaza, Los Angeles CA 90095, USA }
\email{syerikss@math.ucla.edu}

\author[A.V.\! V\"ah\"akangas]{Antti V. V\"ah\"akangas}
\address[A.V.V.]{University of Jyvaskyla, Department of Mathematics and Statistics, P.O. Box 35, FI-40014 University of Jyvaskyla, Finland} 
\email{antti.vahakangas@iki.fi}

\date{\today}

\begin{document}



\begin{abstract}
We prove the  self-improvement of
a pointwise $p$-Hardy inequality. The proof relies on
 maximal function techniques and a characterization of the inequality by curves.
\end{abstract}

\maketitle

\section{Introduction}

Let $X=(X,d,\mu)$ be a metric measure space and
let $1\le p<\infty$. In this paper we are interested in the self-improvement properties of the
\emph{pointwise $p$-Hardy inequality}
\begin{equation}\label{eq.hardy_intro}
|u(x)| \leq  C_{\mathrm{H}}\, d(x,\Omega^c) (\M_{p,\kappa   d(x,\Omega^c)} g (x))\,.
\end{equation}
We say that an open set $\Omega\subsetneq X$ satisfies pointwise $p$-Hardy inequality,
if there are constants $C_{\mathrm{H}}$ and $\kappa$ such that inequality~\eqref{eq.hardy_intro} holds for all 
 $x\in \Omega$ whenever $u$ is a Lipschitz function such that $u=0$ in $\Omega^c=X\setminus \Omega$ and $g$ is 
 a bounded upper gradient of $u$; we refer to
Section~\ref{s.notation} for the definition
of  $\M_{p,\kappa   d(x,\Omega^c)} g(x)$ and the standing assumptions on $X$.
By H\"older's inequality, we see that increasing $p$ will result in a weaker inequality \eqref{eq.hardy_intro}. 
 Self-improvement is concerned with the opposite, and far less intuitive,
possibility of lowering the exponent $p$ slightly. Our main result reads as follows.
Let $1\le  p'<p<\infty$ and assume that $X$ supports
a $p'$-Poincar\'e inequality. 
Assume that  $\Omega$ satisfies a pointwise $p$-Hardy inequality.
Then there exists $q\in (p',p)$ such that  
$\Omega$ satisfies a pointwise $q$-Hardy inequality; we refer to Theorem~\ref{c.main}.
In this paper we  provide  a direct  proof of this self-improvement result with transparent and quantitative bounds for the quantity $p-q>0$ of self-improvement; see Remark \ref{r.quantitative}.

The pointwise $p$-Hardy inequality was first independently studied
by Haj{\l}asz in \cite{MR1458875} and 
by Kinnunen--Martio in \cite{MR1470421}.
Korte et al.\ proved in \cite{MR2854110}
that a pointwise $p$-Hardy inequality characterizes
the so-called uniform $p$-fatness of the complement $\Omega^c$;
we note that uniform $p$-fatness is a uniform $p$-capacity density condition that appears often in potential theory
and PDE's, see e.g.\ \cite{MR1207810}.
Consequently, our proof can be used to show 
the deep self-improvement property
of  uniform $p$-fatness.
This result  was first discovered in Euclidean spaces by Lewis \cite{MR946438}
using potential theoretical arguments.
Subsequently  Mikkonen generalized Lewis' result to the Euclidean weighted setting in his PhD-thesis \cite{MR1386213}. 
Mikkonen's approach, in turn, was  adapted to metric spaces by Bj\"orn et al.\ in \cite{MR1869615}.
This adaptation relies on
the impressive theory of differential structures on
complete (or at least locally complete) metric spaces, 
established by Cheeger in~\cite{MR1708448}.

An alternative approach to the self-improvement
of uniform $p$-fatness   
was recently provided by Lehrb\"ack et al.\ in \cite{MR3673660}.
Their proof builds upon a localization of the
argument due to Koskela--Zhong  \cite{MR1948106} which, in turn, is concerned with the self-improvement
of integral $p$-Hardy inequalities.
Consequently, either one 
of the papers \cite{MR1869615} or \cite{MR3673660} together  with the mentioned characterization in \cite{MR2854110}
can be used to provide an indirect proof of our main result.
In comparison, our  approach is more direct with the additional benefit of  yielding transparent and  quantitative bounds for the self-improvement. Our approach is new even in the classical setting of
Euclidean space equipped with the Euclidean metric and the Lebesgue measure. For a survey on Hardy inequalities, and their connections
to uniform $p$-fatness, we refer to \cite{MR2723821} and references therein. See also \cite{MR2777530}.

The outline of this paper is as follows.
Notation and maximal function techniques are presented in Section \ref{s.notation}.
The pointwise $p$-Hardy inequality is characterized by using curves
in Section~\ref{s.char}.
The actual work for self-improvement via curves is done in Section \ref{e.self_improvement}
and our main results are stated and proved in Section \ref{s.main}.
The main line of our proof is adapted from the paper \cite{SEB} of the first author, where
the self-improvement of a $p$-Poincar\'e inequality is
proved with the aid of maximal functions 
and a characterization  by curves; this result was originally obtained in \cite{MR2415381}
by a different method.
Curiously, the present approach simultaneously explains the self-improvement property of both  $p$-Poincar\'e inequality  and pointwise $p$-Hardy inequality.
We also remark that  Lerner--P\'erez \cite{MR2375698}
established self-improvement properties of Muckenhoupt weights
by a  similar approach to maximal functions. 
It is an open question, to  what extent these ideas can be taken to unify proofs of 
various self-improvement phenomena that are ubiquitous in analysis and PDE.

\medskip

{\bf Acknowledgements.} The authors would like to thank Juha Kinnunen and Juha Lehrb\"ack for their valuable comments.
 The first author is partially supported by the grant DMS\#-1704215 of NSF(U.S.). The first author also thanks 
Enrico Le Donne, Riikka Korte and Juha Kinnunen for hosting and supporting visits at University of Jyv\"askyl\"a and Aalto University during which 
this research was completed.

\section{Notation and auxiliary results}\label{s.notation}

Here, and throughout the paper, we assume that $X=(X,d,\mu)$ is a $C_{\mathrm{QC}}$-quasiconvex  proper metric measure space equipped with a metric $d$ and a 
positive complete  $D$-doubling  Borel
measure $\mu$ such that $\# X\ge 2$, $0<\mu(B)<\infty$ and
\begin{equation}\label{e.doubling}
\mu(2B) \le D\, \mu(B)
\end{equation}
for some $D>1$ and for all balls $B=B(x,r)=\{y\in X\,:\, d(y,x)<r\}$. 
Here we use for $0<\lambda <\infty$ the notation $\lambda B=B(x,\lambda r)$.
The space $X$ is separable
under these assumptions, see \cite[Proposition 1.6]{MR2867756}.
Moreover, the measure  $\mu$ is regular
and, in particular
for every Borel set $E\subset X$ and every $\varepsilon>0$, there
exists an open set $V\supset E$ such that
$\mu(E)\le \mu(V)+\varepsilon$; we refer to \cite[Theorem 7.8]{MR767633} for further details.

We denote by $\Lip(X)$ the space of Lipschitz functions on $X$. That is, we have $u\in \Lip(X)$ iff there
exists a constant $\lambda>0$ such that
\[
\lvert u(x)-u(y)\rvert \le \lambda d(x,y)\,,\qquad \text{\ for all }\ x,y\in X\,.
\]
We let  $\Omega \subset X$ be an open set. We denote by $\Lip_{0}(\Omega)$ 
the space of Lipschitz 
functions on $X$ that vanish on $\Omega^c= X \setminus \Omega$. The set of continous functions on $X$ is denoted by $C(X)$, and $C_0(\Omega) \subset C(X)$ consists of those continuous functions that vanish on $\Omega^c$. We  denote by $LC(X)$  the set  of lower semicontinuous functions on $X$, and 
by $LC_0(\Omega)$ we denote the set of those functions in $LC(X)$
that vanish on $\Omega^c$.  

By a {\em curve} we mean a nonconstant, rectifiable, continuous
mapping from a compact real interval to $X$. By $\Gamma(X)$ we denote the set of all curves in $X$. The
length of a curve $\gamma\in \Gamma(x)$ is  written as
$\len(\gamma)$.
We say that a curve $\gamma\colon [a,b]\to X$ {\em connects}
$x\in X$ to $y\in X$ (or a point $x\in X$ to a set $E\subset X$), if $\gamma(a)=x$ and $\gamma(b)=y$ ($\gamma(b)\in E$, respectively). 
If $x,y \in X$, $E \subset X$ and $\nu \geq 1$ we denote by $\Gamma(X)^\nu_{x,y}$ the set of curves that connect $x$ to $y$ and  whose lengths are at most $\nu d(x,y)$, and by $\Gamma(X)^\nu_{x,E}$ we denote the  set of curves that connect   $x$ to $E$   and whose  lengths are   at most $\nu d(x,E)$. 

We say that a Borel function $g\ge 0$ on $X$ is an 
{\em upper gradient} of a real-valued function $u$ on
$X$ if, for   any curve $\gamma$  connecting any $x\in X$ to any  $y\in X$,    we have
\begin{equation}\label{e.modulus}
\lvert u(x)-u(y)\rvert \le \int_\gamma g\,ds\,.\end{equation}

We use the following familiar notation: \[
u_E=\vint_{E} u(y)\,d\mu(y)=\frac{1}{\mu(E)}\int_E u(y)\,d\mu(y)
\]
is the integral average of $u\in L^1(E)$ over a measurable set $E\subset X$
with $0<\mu(E)<\infty$. Moreover if $E\subset X$, then $\mathbf{1}_{E}$
denotes the characteristic function of $E$; that is, $\mathbf{1}_{E}(x)=1$ if $x\in E$
and $\mathbf{1}_{E}(x)=0$ if $x\in X\setminus E$.
If $1\le p<\infty$ and $u\colon X\to \R$ is a $\mu$-measurable function, then $u\in L^p_{\textup{loc}}(X)$ 
means that for each  $x_0\in X$ there exists $r>0$  such that $u\in L^p(B(x_0,r))$, 
i.e., $\int_{B(x_0,r)} \lvert u(y)\rvert^p\,d\mu(y)<\infty$.

For   $0\le r<\infty$ and  $1\le p<\infty$, and every $f \in L^p_{\mathrm{loc}}(X)$,   we define    the $r$-restricted maximal   function $\M_{p,r} f(x)$  at $x\in X$ by \[\M_{p,r} f(x) \defeq 
\begin{cases} 
\lvert f(x)\rvert\,,\qquad &r=0\,,\\
\displaystyle\sup_{B} \left( \vint_{B} \lvert f(z)\rvert^p ~d\mu(z) \right)^\frac{1}{p}\,,\qquad &r>0\,,
\end{cases}
\]
where the supremum is taken over all balls $B=B(y,t)$ in $X$
such that $x\in B$ and $0<t<r$.

The definition of a pointwise $p$-Hardy inequality is as follows; recall that  $\Omega^c=X\setminus \Omega$.

\begin{definition}\label{def:ptwisePHardy}
 Let $1\le p<\infty$. An open  set  
$\emptyset\not=\Omega \subsetneq X$ is said to satisfy a {\em pointwise $p$-Hardy inequality} if there exists constants $C_{\mathrm{H}}>0$ and  $\kappa  \ge 1$ such that for every Lipschitz function $u \in \Lip_{0}(\Omega)$, every   bounded  upper gradient $g$ of $u$ and  every  $x \in \Omega$, we have
\begin{equation}\label{e.pointwise}
|u(x)| \leq  C_{\mathrm{H}}\, d(x,\Omega^c) (\M_{p,\kappa   d(x,\Omega^c)} g (x))\,.
\end{equation}
\end{definition}

Clearly by H\"older's inequality,  a pointwise $p$-Hardy inequality
implies a pointwise $q$-Hardy inequality for every $1\le p<q<\infty$. 

The following $p$-Poincar\'e inequality
has a corresponding property.

 \begin{definition}
 Let $1\le p < \infty$.
We say that $X$ supports a $p$-Poincar\'e inequality, if 
there are constants   $C_{\mathrm{PI}}>0$   and   $\lambda\ge 1$   such that 
for any   ball $B$ of radius $r>0$   in $X$, any 
  $u\in \Lip(X)$   and any bounded upper gradient $g$ of $u$, we have
\begin{equation}\label{e.poincare}
\vint_{B} \lvert u(x)-u_{B}\rvert\,d\mu(x)
\le C_{\mathrm{PI}} \,r\bigg(\vint_{\lambda B} g(x)^p \,d\mu(x)\bigg)^{1/p}\,.
\end{equation}
Here  $u_B=\vint_B u\,d\mu$.
\end{definition}

We remark that the $p$-Poincar\'e inequality has a self-improving property.
More specifically, a $p$-Poincar\'e inequality for any $1<p<\infty$ implies a $p'$-Poincar\'e inequality for some $p'<p$; see \cite{SEB}
and \cite{MR2415381}.
For a self-contained exposition, we will explicitly assume such an improved Poincar\'e inequality.
The following characterization from \cite[Theorem 1.5]{SEB} 
will be useful.

\begin{lemma}\label{lem:PIchar}
  Let $1\le p<\infty$.  
Then  $X$   supports   a $p$-Poincar\'e inequality if and only if 
there are constants   $C_{\mathrm{A}}>0$, $\nu>C_{\mathrm{QC}}$    and $\kappa\ge 1$ such that, for any non-negative and bounded  $g \in LC(X)$ and any $x,y \in X$,   we have   
\begin{equation}\label{eq:PIchar}
\inf_{\gamma \in \Gamma(X)^\nu_{x,y}} \int_{\gamma} g \,ds \leq   C_{\mathrm{A}}  \,d(x,y) \left( \M_{p,\kappa d(x,y)} g (x) + \M_{p,\kappa d(x,y)} g (y)\right). 
\end{equation}
\end{lemma}

  We need a few auxiliary results involving maximal functions.
 We begin with the following scale invariant weak-type estimate
that is originally from \cite[Lemma 2.3]{SEB}.

\begin{lemma}\label{lem:HLmaxmaxest}
Fix $1\le q<\infty$ and $0<r,s<\infty$. 
Let   $f \in L^q_{\mathrm{loc}}(X)$,  
let $\Lambda>0$, and define  \[ E_{q,s,\Lambda} =\{x\in X \mid \M_{q,s} f(x) > \Lambda\}\,. \]
Then, for every $x \in X$,
\begin{equation}\label{eq:localMaxMax}
\M_{1,r} \mathbf{1}_{E_{q,s,\Lambda}} (x) \leq \frac{D^5 (\M_{q,{r+3s}} f (x))^q}{\Lambda^q}\,.
\end{equation} 
\end{lemma}

\begin{proof}
Fix $x\in X$
and $0<t<r$. 
Let $B=B(y,t)$ be a ball in $X$ such that $x\in B$.
It suffices to prove that 
\begin{equation}\label{e.M_single_ball}
\vint_{B} \mathbf{1}_{E_{q,s,\Lambda}} d\mu
\le \frac{D^5(\mathcal{M}_{q,{r+3s}}f(x))^q}{\Lambda^q}\,.
\end{equation}
The proof of \eqref{e.M_single_ball} is based upon a covering argument.
For each $z\in E_{q,s,\Lambda}\cap B$ we fix
a ball $B_z$ of radius $0<r_{B_z}<s$ such that $z\in B_z$  and
\begin{equation}\label{e.sel}
\bigg(\vint_{B_z}   \lvert f\rvert^q \,d\mu\bigg)^{\frac{1}{q}} >\Lambda\,.
\end{equation}
Suppose that $t<r_{B_z}$ for some 
$z\in E_{q,s,\Lambda}\cap B$. Then
$x\in 3B_z$ and, therefore,
\[
\vint_{B} \mathbf{1}_{E_{q,s,\Lambda}} d\mu
\le 1< \frac{\vint_{B_z}   \lvert f\rvert^q \,d\mu}{\Lambda^q}
\le  \frac{D^2\vint_{3B_z}   \lvert f\rvert^q \,d\mu}{\Lambda^q}
\le \frac{D^2(\mathcal{M}_{q,3s} f(x))^q}{\Lambda^q}\,.
\]
Since  $\mathcal{M}_{q,3s} f(x)\le \mathcal{M}_{q,{r+3s}} f(x)$ and   $D>1$,
we thus obtain inequality \eqref{e.M_single_ball}. Hence in the sequel, we can assume that
$r_{B_z}\le t$ for all $z\in E_{q,s,\Lambda}\cap B$.

By   using   the $5r$-covering lemma \cite[Lemma 1.7]{MR2867756}, we obtain a   countable and   disjoint
family $\mathcal{B}\subset \{B_z \mid z\in E_{q,s,\Lambda}\cap B\}$ of balls such that $E_{q,s,\Lambda}\cap B\subset \bigcup_{B'\in\mathcal{B}} 5B'$.
Hence, by \eqref{e.sel}, 
\begin{align*}
\vint_{B} \mathbf{1}_{E_{q,s,\Lambda}} d\mu
&\le \frac{1}{\mu(B)}\sum_{B'\in\mathcal{B}} \mu(5B')\\
&\le \frac{D^3}{\mu(B)}\sum_{B'\in\mathcal{B}} \mu(B')\\
&\le \frac{D^3}{\Lambda^q \mu(B)}\sum_{B'\in\mathcal{B}}
\int_{B'}  \lvert f\rvert^q \,d\mu\,.
\end{align*}
Since $r_{B'}\le \min\{s,t\}$, we have that $B'\subset B''\defeq B(y, t + 2\min\{s,t\})$ for every $B'\in\mathcal{B}$. Also, $B \subset B'' \subset 3B$, so $\mu(B'') \leq D^2 \mu(B)$.
We can conclude that
\begin{align*}
\vint_{B} \mathbf{1}_{E_{q,s,\Lambda}} d\mu
&\le \frac{D^5}{\Lambda^q}\vint_{B''}   \lvert f\rvert^q \,d\mu\le \frac{D^5(\mathcal{M}_{q,t + 3s}   f (x))^q}{\Lambda^q}\,.
\end{align*}
Since $\mathcal{M}_{q,t+3s} f(x)\le \mathcal{M}_{q,r+3s} f(x)$,
we thus obtain inequality \eqref{e.M_single_ball} also in this case.  
\end{proof} 

 The following approximation lemma is originally from \cite[Lemma 3.7]{SEB}. For the convenience of
the reader, we provide a  proof.
We remark that the regularity of the measure is needed in the proof.
 A Borel function $g\colon X\to [0,\infty)$ is {\em simple}, if
$g=\sum_{j=1}^k a_j\mathbf{1}_{E_j}$ for
some real  $a_j>0$ and Borel sets $E_j\subset X$, $j=1,\ldots,k$.

\begin{lemma}\label{lem:approx}   
Let $1\le p<\infty$.
Let $g\colon X\to [0,\infty)$ be a simple Borel function. Then, for each $x\in X$ and every $\varepsilon>0$, there exists   a non-negative and bounded $g_{x,\varepsilon}\in LC(X)$ such that
  $g(y)\le g_{x,\varepsilon}(y)$ for all  $y\in X\setminus \{x\}$  and
$\M_{p,  r }g_{x,\varepsilon}(x) \leq \M_{p,  r }g(x)+\varepsilon$ if $r>0$.  
\end{lemma}

\begin{proof} 
We prove the claim, while assuming that $\diam(X)=\infty$. The case $\diam(X)<\infty$ is similar, and we omit the modifications. Fix $x\in X$ and  $\varepsilon>0$.
In the first step, we prove an auxiliary statement for  a Borel set  $E\subset X$.
Namely,  we will show that
there exists an open set $U\subset X$ such that $\mathbf{1}_E\le \mathbf{1}_U$ in $X\setminus \{x\}$ and 
\begin{equation}\label{e.verysimple}
\M_{p,  r }(\mathbf{1}_U-\mathbf{1}_E)(x)<\varepsilon\,,\qquad \text{ if } r>0\,.
\end{equation}
To prove this auxiliary statement, for each $m\in \Z$, we write 
\[A_m=\{y\in X\,:\, 2^{m-1}<d(x,y)<2^{m+1}\}\,.\]
Observe that each $y\in X$ belongs to at most two  annuli.
We also have that $\mu(A_m)>0$, since $X$ is  connected and unbounded. 
Hence, if $m\in\Z$ then by regularity of the measure $\mu$, there is an  open set $U_m\subset A_m$
such that
\begin{equation}\label{e.measures}
A_m\cap E\subset U_m\qquad \text{\ and }\ \qquad \mu(U_m\setminus E)=\mu(U_m\setminus (A_m\cap E))<  \frac{\varepsilon^{p}}{2D^4} \mu(A_m)\,.
\end{equation}
Define $U=\bigcup_{m\in\Z} U_m$. Then
\begin{equation}\label{e.inclusion}
E\setminus \{x\}\subset \bigcup_{m\in\Z} (A_m\cap E)\subset \bigcup_{m\in\Z} U_m=U\,.
\end{equation}
As a consequence, we then have  $\mathbf{1}_E(y)\le \mathbf{1}_U(y)$ for every $y\in X\setminus \{x\}$. 
To prove \eqref{e.verysimple}, we let $r>0$ and let $B(y,t)\subset X$ be a ball in $X$ such
that $x\in B(y,t)$ and $0<t<r$.  Then
$\mathbf{1}_U-\mathbf{1}_E=\mathbf{1}_{U\setminus E}$ almost everywhere, and therefore
by \eqref{e.measures} we get
\begin{align*}
\vint_{B(y,t)} \lvert \mathbf{1}_U-\mathbf{1}_E\rvert^p\,d\mu&=\vint_{B(y,t)} \mathbf{1}_{U\setminus E}\,d\mu\\
&\le \frac{1}{\mu(B(y,t))} \int_X\sum_{m=-\infty}^{\lceil \log_2 (2t)\rceil} \mathbf{1}_{U_m\setminus E}\,d\mu\\
&= \frac{\varepsilon^{p}}{2D^4\mu(B(y,t))}\sum_{m=-\infty}^{\lceil \log_2 (2t)\rceil} \mu(A_m)\\
&\le \frac{\varepsilon^{p}}{D^4}\frac{\mu(B(x,8t))}{\mu(B(y,t))} \le\varepsilon^{p}\frac{\mu(B(x,t))}{\mu(B(y,2t))}\le \varepsilon^{p}\,.
\end{align*}
By raising this estimate to power $\frac{1}{p}$ and then taking
supremum over all balls $B(y,t)$ as above, we obtain 
inequality \eqref{e.verysimple}.

We now turn to the proof of the actual lemma. Let $g=\sum_{j=1}^k a_j\mathbf{1}_{E_j}$
be such that $a_j>0$ and  $E_j\subset X$ is a Borel set for each $j=1,\ldots,k$.
By the auxiliary statement,  for each $j=1,\ldots,k$, there exists a non-negative and bounded 
$g_{x,\varepsilon,j}\in  LC(X)$ such that $\mathbf{1}_{E_j}\le g_{x,\varepsilon,j}$ in $X\setminus \{x\}$
and 
\begin{equation}\label{e.error}
\M_{p,  r }(g_{x,\varepsilon,j}-\mathbf{1}_{E_j})(x) \le\frac{\varepsilon}{k\max_j a_j }\,.
\end{equation}
Now we define $g_{x,\varepsilon}=\sum_{j=1}^k a_jg_{x,\varepsilon,j}$.
Then $g\le g_{x,\varepsilon}$ in $X\setminus \{x\}$. Moreover,  by the subadditivity and positive homogeneity of the maximal function,
and inequalities \eqref{e.error}, we have
\begin{align*}
\M_{p,r}g_{x,\varepsilon}(x)&=\M_{p,r}(g+g_{x,\varepsilon}-g)(x)\\&\le \M_{p,r}g(x)
+\M_{p,r}(g_{x,\varepsilon}-g)(x)\\&\le \M_{p,r}g(x)
 + \sum_{j=1}^k a_j\M_{p,r}(g_{x,\varepsilon,j}-\mathbf{1}_{E_j})(x)\\&\le \M_{p,r}g(x)+\varepsilon\,.
\end{align*}
This concludes the proof.
 \end{proof}

\section{Characterization  by curves}\label{s.char}

  We   translate the   pointwise    $p$-Hardy inequality to an equivalent problem of accessibility.   This problem   can be phrased as a problem of finding a single curve with a small integral. 
The standing assumptions concerning the space $X$ are stated in
Section \ref{s.notation}.

\begin{lemma}\label{lem:HardyChar} 
  Let   $1\le p<\infty$.
Then an open set $\emptyset\not=\Omega \subsetneq X$  satisfies a pointwise $p$-Hardy inequality if, and only if, there are constants   $C_\Gamma>0$, $  \nu >C_{\mathrm{QC}}$ and $\kappa\ge 1 $ such that for   each   non-negative and bounded $g \in   LC(X) $  and   every   $x \in \Omega$, we have
 \begin{equation}\label{eq:HardyChar} 
\inf_{\gamma \in \Gamma(X)^\nu_{x,\Omega^c}} \int_{\gamma} g \,ds \leq C_\Gamma \,d(x,\Omega^c) \left( \M_{p,\kappa d(x,\Omega^c)} g (x) \right).
\end{equation}
\end{lemma}

\begin{proof}
Throughout this proof, we tacitly assume that
  curves   are parametrized by   arc   length. 
First suppose that   an open set $\emptyset \not=\Omega\subsetneq X$   satisfies a pointwise $p$-Hardy inequality \eqref{e.pointwise} with constants
$C_{\mathrm{H}}>0$ and $\kappa_\Gamma>1$.  
  Fix a  non-negative and bounded function  $g\in LC(X)$.   
Fix   $x\in \Omega$   and let $\delta>0$. 

Define    a function  $u\colon X\to [0,\infty)$ by setting  
\begin{equation}\label{e.u_def}
  u(y)   = \inf_{\gamma} \int_\gamma h \,ds\,,\qquad   y\in X\,,  
\end{equation}
where 
$h = g + \M_{p,  \kappa_\Gamma   d(x,\Omega^c)} g (x)+\delta$,   which is a non-negative bounded Borel function,  
and the infimum is taken over all curves
$\gamma$ in $X$   connecting $y$ to $\Omega^c$.   Let us
remark that these curves are not
subject to any distance conditions. 
Clearly, we have that
$u=0$ in $\Omega^c$.
  Fix $y,w\in X$   and consider any curve $\sigma$   connecting $y$ to $w$.  
We claim that
\begin{equation}\label{e.des}
\lvert u(  y )-u(  w )\rvert\le \int_\sigma h\,ds\,.
\end{equation}
From this it follows, in particular, that $h$ is an upper gradient of $u$. Moreover, since $X$ is quasiconvex and $h$ is bounded, it follows from \eqref{e.des} that $u\in \Lip_0(\Omega)$.

In order to prove  \eqref{e.des}, we may
assume that $u(y)>u(w)$. Fix $\eps>0$ and let $\gamma$   be a curve in $X$ such that
connects $w$ to $\Omega^c$ and   satisfies inequality  
\[
u(w)\ge \int_\gamma h\,ds - \eps\,.
\]
Let $\sigma \gamma$ be the concatenation of $\sigma$ and $\gamma$. Then
\begin{align*}
\lvert u(y)-u(w)\rvert &=u(y)-u(w)\\
&\le \int_{ \sigma\gamma } h\,ds - \int_\gamma h\,ds + \eps=\int_{\sigma} h\,ds+\eps\,.
\end{align*}
The desired inequality \eqref{e.des}
follows by taking $\eps\to 0_+$.

Now, applying the assumed pointwise $p$-Hardy inequality \eqref{def:ptwisePHardy}
to the function $u$ and to its   bounded    upper gradient $h$ yields
\[
u(x)
\le C_{\mathrm{H}}\,d(x,\Omega^c)(\mathcal{M}_{p,\kappa_\Gamma d(x,\Omega^c)} h(x))<\infty\,.
\]
Since $u(x)\ge \delta d(x,\Omega^c)>0$, by \eqref{e.u_def} there is a curve $\gamma$ in $X$ 
  connecting $x$ to $\Omega^c$ such that  
\begin{equation}\label{e.strong}
\begin{split}
\int_\gamma g\,ds + (\mathcal{M}_{p,  \kappa_\Gamma   d(x,\Omega^c)} g(x)+\delta)\len(\gamma)
 &= \int_\gamma h\,ds   \le 2u(x)\\
&\le 2C_{\mathrm{H}}\,d(x,\Omega^c)(\mathcal{M}_{p,  \kappa_\Gamma  d(x,\Omega^c)} h(x))
\\&\le 2C_{\mathrm{H}}\,d(x,\Omega^c)(2\mathcal{M}_{p,  \kappa_\Gamma  d(x,\Omega^c)} g(x)+\delta)\,.
\end{split}
\end{equation}
  The last inequality  follows from the sublinearity of maximal function. 
We can now   conclude from \eqref{e.strong} that
$\len(\gamma)\le 4C_{\mathrm{H}}\,d(x,\Omega^c)$.
By taking $\delta\to 0_+$, we obtain from \eqref{e.strong} that
\[
\int_\gamma g\,ds\le 
4C_{\mathrm{H}}\,d(x,\Omega^c)(\mathcal{M}_{p,  \kappa_\Gamma   d(x,\Omega^c)} g(x))\,.
\]
Thus, inequality
\eqref{eq:HardyChar} holds   with 
\[C_\Gamma=4C_{\mathrm{H}}\,,\qquad \kappa=\kappa_\Gamma\,,\qquad \nu>\max\{C_{\mathrm{QC}},4C_{\mathrm{H}}\}\,.\]

  For the converse implication,   we assume that inequality \eqref{eq:HardyChar} holds,   for all    non-negative and bounded    $g\in LC(X)$, and for all $x\in\Omega$. 
We need to prove that $\Omega$ satisfies a pointwise $p$-Hardy inequality.
To this end, we let $u\in \Lip_0(\Omega)$ and let $g$
be a bounded  upper gradient of $u$.
We also fix $x\in \Omega$. 
Since $g$ is not necessarily lower semicontinuous, some approximation is first needed so that we can get to apply \eqref{eq:HardyChar} and thereby establish  inequality \eqref{e.pointwise}.

Let
 $(g_N)_{N\in\N}$ be a pointwisely increasing sequence of non-negative simple Borel functions such that $\lim_{N\to\infty} g_N=g$ uniformly in $X$. Fix $\varepsilon>0$. 
 By the uniform convergence, there exists $N\in\N$ such that
 for all $\gamma\in \Gamma(X)^\nu_{x,\Omega^c}$ we have
\begin{equation}\label{e.simple}
\begin{split}
 \int_{\gamma} g \,ds &=\, \int_{\gamma} g_N \,ds + \int_{\gamma} (g-g_N) \,ds \\
 &\le 
 \int_{\gamma} g_N \,ds + \sup_{y\in X}(g(y)-g_N(y))\len(\gamma)\\
  &\le 
 \int_{\gamma} g_N \,ds + \sup_{y\in X}(g(y)-g_N(y))\nu d(x,\Omega^c)\\
 &\le  \int_{\gamma} g_N \,ds+\varepsilon\,.
 \end{split}
\end{equation}
Let $g_{N,x,\varepsilon}\in LC(X)$ be 
the non-negative bounded  approximant of $g_N$ given by Lemma \ref{lem:approx}.
By 
inequality \eqref{eq:HardyChar} and Lemma \ref{lem:approx},
there exists $\gamma_{N}\in \Gamma(X)^\nu_{x,\Omega^c}$ such that
\begin{align*}
\int_{\gamma_{N}} g_{N,x,\varepsilon} \,ds&\leq C_\Gamma d(x,\Omega^c) \left( \M_{p,\kappa d(x,\Omega^c)} g_{N,x,\varepsilon} (x) \right)+\varepsilon \\&\leq \ C_\Gamma d(x,\Omega^c) \left( \M_{p,\kappa d(x,\Omega^c)} g_{N} (x)+\varepsilon \right)+\varepsilon\\&\leq C_\Gamma d(x,\Omega^c) \left( \M_{p,\kappa d(x,\Omega^c)} g (x)+\varepsilon \right)+\varepsilon\,.
\end{align*}
 
  Without loss of generality, we may assume that $\gamma_N(t)=x$ only if $t=0$. 
On the other hand, by Lemma \ref{lem:approx}, we have  $g_{N}\le g_{N,x,\varepsilon}$ in $X\setminus \{x\}$. Inequality \eqref{e.simple}, with $\gamma=\gamma_N$, implies that
\begin{align*}
 \int_{\gamma_N} g \,ds &\leq   \int_{\gamma_N} g_N \,ds + \varepsilon\\
 &\le  \int_{\gamma_N} g_{N,x,\varepsilon} \,ds + \varepsilon\\
&\le  C_\Gamma d(x,\Omega^c) \left( \M_{p,\kappa d(x,\Omega^c)} g (x)+\varepsilon \right)+2\varepsilon\,. 
\end{align*}
Since $g$ is an upper gradient of $u\in \Lip_0(\Omega)$, we get
\begin{align*}
 \lvert u(x)\rvert 
 &=\lvert u(\gamma_N(0))-u(\gamma_{N}(\len(\gamma_{N})))\rvert\\&
 \leq \int_{\gamma_N} g \,ds\le C_\Gamma d(x,\Omega^c) \left( \M_{p,\kappa d(x,\Omega^c)} g (x)+\varepsilon \right)+2\varepsilon\,,
\end{align*} 
 and   letting $\varepsilon \to 0$ gives the pointwise $p$-Hardy inequality \eqref{e.pointwise}
   with $C_{\mathrm{H}}=C_\Gamma$ and $\kappa\ge 1$.  
\end{proof}

While seemingly technical, the task of infimizing 
  in \eqref{eq:HardyChar}   is often reduced to constructing an explicit curve,  for which the upper  bound holds.  
In particular, our proof for  self-improvement of pointwise Hardy inequalities is based on establishing the existence of such a single curve for some exponent $q<p$.

Next we define  a convenient $\alpha$-function that condenses the pointwise $p$-Hardy inequality, or  inequality \eqref{eq:HardyChar} to be more specific,  in a single function at the expense of abstraction. 
Indeed, the following definition looks complicated at first sight, but for our purposes the quantity 
$\alpha_{p,\Omega}$ is precisely the correct
way to express the pointwise $p$-Hardy inequality.

\begin{definition}
Let $\emptyset\not=\Omega\subsetneq X$ be an open set.
If   $\tau\ge 0$,   $\kappa,p\ge 1$ and $x\in \Omega$, we write
\[\mathcal{E}^{\kappa,\tau}_{p,x,\Omega} = \{g \in   LC(X)   \mid  \M_{p,  \kappa   d(x,\Omega^c)} g (x) \leq \tau \text{ and }g(y) \in [0,1]\text{ for all }y\in X\}\,.\]
If also   $\nu>C_{\mathrm{QC}}$,   then we write   \begin{equation}\label{eq:alphaDef}
\alpha_{p,\Omega}(  \nu,\kappa,\tau) \defeq \sup_{x \in \Omega} \sup_{g \in \mathcal{E}^{\kappa,\tau}_{p,x,\Omega}} \frac{ \inf_{\gamma \in \Gamma(X)^\nu_{x,\Omega^c}}\int_\gamma g \,ds}{d(x,\Omega^c)}\,. 
\end{equation}
\end{definition}

Concerning definition \eqref{eq:alphaDef}, the parameter   $\nu$   is related to the maximum length of the curves $\gamma$ that are used so that $\len(\gamma)\le \nu d(x,\Omega^c)$.   The parameters   $\kappa$ and $\tau$ measure the non-locality and size of  the maximal function $\M_{p,\kappa d(x,\partial \Omega)} g (x)$, respectively.
%

  The 
fundamental connection between inequality \eqref{eq:HardyChar}
and the $\alpha$-function is established in the following lemma.   

\begin{lemma} \label{lem:rewrite}
  Let $\emptyset\not=\Omega\subsetneq X$ be an open set, and let $\kappa,p\ge 1$ and $\nu>C_{\mathrm{QC}}$.  
Let   $g\in LC(X)$   be such that $g(y)\in [0,1]$ for
every $y\in \Omega$.
Then,
for every $x\in \Omega$, we have 
 \begin{equation}
\inf_{\gamma \in \Gamma(X)^\nu_{x,\Omega^c}} \int_{\gamma} g \,ds \leq d(x,\Omega^c) \alpha_{p,\Omega}\left(  \nu,\kappa,\left( \M_{p,\kappa d(x,\Omega^c)} g (x) \right) \right). 
\end{equation}
\end{lemma}

\begin{proof}
Fix   $g\in LC(X)$   such that $g(y)\in [0,1]$ for all $y\in X$. 
Let $x\in\Omega$ and write 
\[\tau=\M_{p,\kappa d(x,\Omega^c)} g (x)   \ge 0 \,.\]
Then $g\in \mathcal{E}^{\kappa,\tau}_{p,x,\Omega}$, and hence
\[
\frac{\inf_{\gamma \in \Gamma(X)^\nu_{x,\Omega^c}} \int_{\gamma} g \,ds}{d(x,\Omega^c)}\le
\sup_{h \in \mathcal{E}^{\kappa,\tau}_{p,x,\Omega}} \frac{ \inf_{\gamma \in \Gamma(X)^\nu_{x,\Omega^c}}\int_\gamma h \,ds}{d(x, \Omega^c)}
\le \alpha_{p,\Omega}(  \nu,\kappa,\tau)
\]
Where the last step follows, since $x\in \Omega$.  
\end{proof}

In particular, from Lemma \ref{lem:rewrite} we now obtain the following sufficient condition for the pointwise
$p$-Hardy inequality in terms of a $\tau$-linear upped bound for the $\alpha$-function.

\begin{lemma}\label{lem:alphaClass}
Let $1\le p<\infty$ and let $\emptyset\not=\Omega\subsetneq X$ be an open set.
Suppose that there
are constants $  \nu >C_{\mathrm{QC}}$, $\kappa\ge 1$   and $C_\alpha >0$   such that,
for any   $\tau\ge 0$,   we have
\[
\alpha_{p,\Omega}(  \nu,\kappa,\tau)\le C_\alpha \tau\,.
\]
Then   $\Omega$   satisfies a pointwise $p$-Hardy inequality.
\end{lemma}

\begin{proof}
  By Lemma \ref{lem:HardyChar}, it suffices to find
a constant $C_\Gamma>0$ such that  inequality \eqref{eq:HardyChar} holds for each non-negative bounded $g\in LC(X)$ and every $x\in\Omega$ --- the remaining constants $\nu$ and $\kappa$ are given in the assumptions of the present lemma.
Fix such a function $g$ and a point $x\in\Omega$. Since
$g$ is bounded and inequality \eqref{eq:HardyChar} is invariant under multiplication
of $g$ with a strictly positive constant, we may further assume that
$g(y)\in [0,1]$ for all $y\in X$.
 
Then the desired estimate \eqref{eq:HardyChar},
with $C_\Gamma=C_\alpha$, follows immediately 
 from Lemma \ref{lem:rewrite} and the assumptions.
\end{proof}

The converse of Lemma \ref{lem:alphaClass} is also true, as we will
see in Section \ref{e.self_improvement}.
Therein the following inequalities for the $\alpha$-function become useful.

\begin{lemma}\label{lem:alphaEst} 
  Let $\emptyset\not=\Omega\subsetneq X$ be an open set.  
  Let  $0\le \tau<\tau'$, $\kappa,p\ge 1$ and $\nu>C_{\mathrm{QC}}$.  
Then
\[
  \alpha_{p,\Omega}(\nu,\kappa,\tau)  \leq \alpha_{p,\Omega}(\nu,\kappa,\tau')\,,\qquad 
\alpha_{p,\Omega}(\nu,\kappa,\tau) \leq \nu\,,
\]
and, for   every   $M\geq 1$, 
\[\alpha_{p,\Omega}(\nu,\kappa,M\tau) \leq M\alpha_{p,\Omega}(\nu,\kappa, \tau)\,.
\]
\end{lemma}

\begin{proof}
These   inequalities   are clear from the definition \eqref{eq:alphaDef}. In this connection, it is important to observe that  $g$ is bounded by $1$
  and $\len(\gamma)\le \nu d(x,\Omega^c)$.  
\end{proof}

\section{Key theorem for self-improvement}\label{e.self_improvement}

In this section we  formulate and prove our key  Theorem \ref{t.converse_and_impro}.
In the light of Lemma \ref{lem:alphaClass},  Theorem \ref{t.converse_and_impro} implies
self-improvement of pointwise $p$-Hardy inequalities; see
Theorem \ref{c.main}.
This theorem also provides a
converse of Lemma \ref{lem:alphaClass} for  $p>1$;
see Theorem \ref{t.consequences}.

Lemmata \ref{lem:PIchar} and \ref{lem:HardyChar}  give us the proper tools 
for the proof of Theorem~\ref{t.converse_and_impro}.
We assume that $X$  supports a better $p'$-Poincar\'e inequality
for some $p'<p$. This assumption allows us to focus on the new phenomena
that arise especially in connection with the self-improvement of pointwise $p$-Hardy inequalities.

\begin{theorem}\label{t.converse_and_impro}
  Let   $1\le  p'<p<\infty$.   Assume that $X$ supports
a $p'$-Poincar\'e inequality. 
Let $\emptyset\not=\Omega\subsetneq X$ be an open set
that satisfies a pointwise $p$-Hardy inequality.
Then there exists an exponent    $q\in (p',p)$   and constants 
  $  N   >C_{\mathrm{QC}}$, $K\ge 1$ and $  C_\alpha>0 $   such that
\begin{equation}\label{e.main_dest}
\alpha_{q,\Omega}(  N,K, \tau)\le C_\alpha\tau
\end{equation}
whenever  $\tau  \ge   0$. 
\end{theorem}

\begin{proof}
 By H\"older's inequality, we can assume that $\max\{1,p/2\}\le p'$. This
assumption allows us to choose  $M$ below
independent of $p$. This property, in turn, is beneficial in Remark \ref{r.quantitative}, where
a quantitative analysis is performed.
  Since    $\Omega$ satisfies a pointwise $p$-Hardy inequality, by Lemma \ref{lem:HardyChar} it satisfies  inequality  \eqref{eq:HardyChar} with constants   $C_\Gamma>0$, $\nu_\Gamma>C_{\mathrm{QC}}$  and $\kappa_\Gamma\ge 1$.   Also,   let $C_{\mathrm{A}}>0$, $\nu_{\mathrm{A}} >C_{\mathrm{QC}}$ and $\kappa_{\mathrm{A}}\ge 1$   be the constants from  inequality  \eqref{eq:PIchar} in Lemma \ref{lem:PIchar},   for the exponent $p'<p$. 
Without loss of generality, we may  assume that 
  $\kappa_\Gamma=\kappa_{\mathrm{A}}=:\kappa$
and that $\nu_\Gamma=\nu_{\mathrm{A}}=:\nu$.  

It suffices to prove that there exists  $k \in \N$, 
$  K,S\in [1,\infty)$, $N \in (C_{\mathrm{QC}},\infty)$,   $M>1$   and  $\delta \in (0,1)$ such that,
   for each $q\in (p',p)$ and every $\tau>0$, we have  
 \begin{equation}\label{eq:iterationEst}
     \alpha_{q,\Omega}(N,  K ,\tau) \leq S \tau + \delta \max_{i=1, \dots, k} \big(M^{-iq/p} \alpha_{q,\Omega}(N,K,M^i \tau)\big)\,.
 \end{equation}
Indeed, from this inequality and Lemma \ref{lem:alphaEst}, we get
\[\alpha_{q,\Omega}(N,K,\tau) \leq S \tau + \delta M^{k\frac{p-q}{p}} \alpha_{q,\Omega}(N,K,\tau)\qquad   \text{ for all }q\in (p',p)\text{ and }\tau>0 \,.
\]
  In order to absorb the last term on the right to the left,   we need   $\delta M^{k\frac{p-q}{p}} < 1$.   This    can be ensured  by  choosing   $q\in (p',p)$   so close to $p$ that \[0<p-q < \frac{p\ln(\frac{1}{\delta})}{k\ln(M)}\,.\]
   With this choice  of $q$ we find for all $\tau>0$ that
\[\alpha_{q,\Omega}(N,K,\tau) \leq \left(\frac{S}{1-\delta M^{k\frac{p-q}{p}}}\right)\tau=:C_\alpha \tau\,.\]
Then, this inequality holds also for $\tau=0$, which is seen  by using monotonicity property  of the $\alpha$-function, see Lemma \ref{lem:alphaEst}. Thus, the desired inequality \eqref{e.main_dest} follows from \eqref{eq:iterationEst}.
 Hence, we are left with proving inequality \eqref{eq:iterationEst}. 
 
 At this stage, we fix the auxiliary parameters
\[
  K=4\kappa\,,\quad N=3\nu\,,\quad M = 4\,,\quad \delta =  \frac{1}{4}\,. \]
 We also fix  $k\in\N$ so large that $  C_{\Gamma}^p  \frac{2^p D^5}{k^{p-1}}<\delta^p$, that is,\ $k> ( 2^{p}\delta^{-p} C_{\Gamma}^p   D^5 )^{\frac{1}{p-1}}$. 
The last auxiliary parameter is defined to be
$S= 1+M^{k}  \nu +  3 C_{\mathrm{A}} M^{k}$.  
We also let
 $q\in (p',p)$ and $\tau>0$.
Now, the overall strategy is as follows: we will construct, for any   $x\in\Omega$   and any $g \in \mathcal{E}^{K,\tau}_{q,x,\Omega}$, a curve $\gamma \in \Gamma(X)^N_{x,\Omega^c}$  such that 
\begin{equation}\label{e.desired_curve}
\int_\gamma g \,ds \leq S \tau d(x, \Omega^c) + \delta \max_{i=1, \dots, k} \big(M^{-i q/p} \alpha_{q,\Omega}(N,K,M^i \tau)\big) d(x, \Omega^c) \,.
\end{equation}
Dividing both sides of this estimate by $d(x,\Omega^c)$, and then taking the supremum over $x$ and $g$ as above,  proves inequality  \eqref{eq:iterationEst}.

Let us fix $x\in\Omega$ and $g \in \mathcal{E}^{K,\tau}_{q,x,\Omega}$. For each $i\ge 1$, we write 
 \begin{align*}
 E_i :&= \{z \in \Omega  \mid  \M_{q,\kappa   d(x,\Omega^c)} g(z) > M^i \tau\}\,,
 \end{align*} and define a bounded function $h\colon X\to [0,\infty)$ by setting 
\[
h = \frac{1}{k}\sum_{i=1}^k \mathbf{1}_{E_i} M^{iq/p}\,.
\]
Since   $E_j \supset E_i$   if $j\le i$
 and $p/2\le p'<q<p$,  we have
\begin{align*}
h^p &\leq \frac{1}{k^p} \sum_{j=1}^k \bigg(\sum_{i=1}^j M^{iq/p}\bigg)^p \mathbf{1}_{E_j} \le \frac{2^p}{k^p} \sum_{j=1}^k \mathbf{1}_{E_j} M^{jq}\,.
\end{align*}
In the final estimate, we also use the equation $M=4$ to obtain the factor $2^p$.
Observe that  $\mathbf{1}_{E_i}\in LC_0(\Omega)$ 
 since $E_i$ is open, for each $i=1,\ldots,k$. Hence, we have
 $h\in LC_0(\Omega)  \subset LC(X) $. 
 By  sublinearity and monotonicity of the maximal function,  Lemma \ref{lem:HLmaxmaxest}, and 
the assumption that $g \in \mathcal{E}^{K,\tau}_{q,x,\Omega}$,   where  $K=4\kappa$,   we obtain
  \begin{equation}\label{e.essential}
  \begin{split}
    \left(\M_{p,\kappa d(x,\Omega^c)} h (x) \right)^p &  \leq \frac{2^p}{k^p}\sum_{j=1}^k (\M_{1,\kappa d(x,\Omega^c)} \mathbf{1}_{E_j}(x)) M^{jq} \\
						       &  \leq \frac{2^p D^5}{k^p}\sum_{j=1}^k \frac{(\M_{q, 4\kappa  d(x,\Omega^c)} g(  x ))^q}{M^{jq}\tau^q} M^{jq}\leq \frac{2^p D^5}{k^{p-1}}\,.
\end{split}			       
\end{equation}
Then, by the choice of $k$ and estimate \eqref{e.essential}, we obtain that 
 \[
   C_\Gamma \M_{p,\kappa d(x,\Omega^c)} h (x) < \delta\,,
 \]
and therefore from Lemma \ref{lem:HardyChar} with exponent $p$ we obtain a 
curve $\gamma_0 \in \Gamma(X)^{\nu}_{x,\Omega^c}$, which is parametrized by arc length, such that
 \begin{equation}\label{e.h_app}
 \int_{\gamma_0} \frac{1}{k}\sum_{i=1}^k \mathbf{1}_{E_i} M^{iq/p} \,ds = \int_{\gamma_0} h\,ds \leq \delta d(x,\Omega^c)\,,
 \end{equation}
 and
 \begin{equation}\label{e.gamma0_Lestimate}
\len(\gamma_0) \leq   \nu  d(x,\Omega^c)\,.
\end{equation}
Clearly, without loss of generality, we may also assume that $\gamma_0([0,\len(\gamma_0)))\subset \Omega$. 

By inequality \eqref{e.h_app}, there  exists $i_0 \in \{1,\ldots,k\}$  such that
\begin{equation}\label{e.apriori}
\int_{\gamma_0} \mathbf{1}_{E_{i_0}}\,ds \leq \delta M^{-i_0q/p} d(x,\Omega^c)\,.
\end{equation}
Let $O = \gamma_0^{-1}(E_{i_0})$ and denote $  T   = [0, \len(\gamma_0)] \setminus O$.
By the lower semicontinuity of $g$   and the definition of $E_{i_0}$   we have, for all $t \in T\setminus \{\len(\gamma_0)\}$,
\begin{equation}\label{e.K_estimate}
g(\gamma_0(t))\le \M_{q,\kappa d(x,\Omega^c)} g(\gamma_0(t))\le  M^{i_0} \tau\,. 
\end{equation} 
  Since $E_{i_0}$ is open in $X$,   the set $O$ is relatively open
  in $[0,\len(\gamma_0)]$.   Observe that $0\not\in O$ since $g\in \mathcal{E}^{K,\tau}_{q,x,\Omega}$ and $K>\kappa$.
 Likewise $\len(\gamma_0)\not\in O$ since $\gamma_0(\len(\gamma_0))\in \Omega^c$.  
  There are now essentially  two   different cases to be handled; the remaining cases of corresponding finite unions   are treated
in a similar way.   Namely, the two cases are:  
 \begin{equation}\label{e.cases}
O = \bigcup_{i \in \N} (a_i,b_i)\quad \text{ or }\quad O= (a_0,b_0)\cup  \bigcup_{i \in \N} (a_i,b_i)\,.
\end{equation}
  The second case takes place,  if 
 there exists $0<t_0<\len(\gamma_0)$ such that
$\gamma_0(t)\in E_{i_0}$ for every $t_0<t<\len(\gamma_0)$.   
In both cases 
the intervals   (called `gaps')   are pairwise disjoint and
$a_i<b_i<\len(\gamma_0)$ for each $i\in\N$, and in the   second    case $a_0<b_0=\len(\gamma_0)$.   Moreover, in both cases $\gamma_0(a_i),\gamma_0(b_i)\in \Omega\setminus E_{i_0}$ for each $i\in \N$,
and in the second case $\gamma_0(a_0)\in \Omega\setminus E_{i_0}$.
We remark that in the second case  
$\gamma_0(b_0)\not\in \Omega\setminus E_{i_0}$, and this
special property of the `final gap' $(a_0,b_0)$ 
distinguishes it from the remaining gaps.
Write $d_i\defeq d(\gamma_0(a_i),\gamma_0(b_i))$  for each $i$. Then, by inequality \eqref{e.apriori}, we have
\begin{equation}\label{e.d_i_estimates}
\sum_i d_i \leq \sum_i \len(\gamma_0\vert_{[a_i,b_i]})=
\sum_i \int_{\gamma_0\vert_{[a_i,b_i]}} \mathbf{1}_{E_{i_0}}\,ds
\le \int_{\gamma_0} \mathbf{1}_{E_{i_0}}\,ds \le 
\delta M^{-i_0 q/p} d(x,\Omega^c)\,.
\end{equation}
There are now two cases to be treated  in a case study.
 
Let us first consider the case $O = \bigcup_{i \in \N} (a_i,b_i)$. Fix $i\in \N$. Since   $\gamma_0(a_i), \gamma_0(b_i)\in   \Omega\setminus E_{i_0} $, there holds
\begin{equation}\label{e.M_q}
\M_{q,\kappa d(x,\Omega^c)} g(\gamma_0(a_i)) \leq M^{i_0} \tau\quad \text{ and }\quad \M_{q,\kappa d(x,\Omega^c)} g(\gamma_0 (b_i)) \leq M^{i_0} \tau\,.
\end{equation}
Lemma \ref{lem:PIchar} applied to the  $p'$-Poincar\'e inequality,   and to the   two points $\gamma_0(a_i)$ and $ \gamma_0 (b_i)$,  provides
us with a curve    $\gamma^i \co [a_i,b_i] \to X$ such that
$\gamma^i(a_i)= \gamma_0(a_i)$, $\gamma^i(b_i)= \gamma_0(b_i)$, 
\begin{equation}\label{eq:lenestgammai}
\len(\gamma^i) \leq   \nu   d(\gamma_0(a_i), \gamma_0(b_i)) =   \nu  d_i\,,
 \end{equation}
and,   by using also the fact that $p'< q$ and H\"older's inequality,   
\begin{equation}\label{e.i_curve}
\begin{split}
    &\int_{\gamma^i} g \,ds\\      & \leq   C_{\mathrm{A}}  d(\gamma_0(a_i),\gamma_0(b_i))\left( \M_{q,\kappa d(\gamma_0(a_i),\gamma_0(b_i))} g (\gamma_0(a_i)) + \M_{q,\kappa d(\gamma_0(a_i),\gamma_0(b_i))} g (\gamma_0(b_i))\right)\\&\qquad +
    \underbrace{C_{\mathrm{A}}  d(\gamma_0(a_i),\gamma_0(b_i))M^{i_0} \tau}_{>0}\,. 
        \end{split}
    \end{equation} 
  We observe that $\kappa d(\gamma_0(a_i),\gamma_0(b_i)) \leq \kappa d(x,\Omega^c)$,
 which follows from  \eqref{e.d_i_estimates} since
\[d(\gamma_0(a_i),\gamma_0(b_i))=d_i   \le \sum_i d_i\leq d(x,\Omega^c)\,.\]
This estimate together with \eqref{e.M_q} and \eqref{e.i_curve} yields  
   \begin{equation}\label{eq:intestgammai}
   \begin{split}
 \int_{\gamma^i} g \,ds   & \leq
    3 C_{\mathrm{A}} d(\gamma_0(a_i),\gamma_0(b_i)) M^{i_0} \tau = 3 C_{\mathrm{A}} M^{i_0} \tau d_i\,.
    \end{split}    
\end{equation}
  Let us now   define   a curve   $\gamma \co [0, \len(\gamma_0)] \to X$ by  setting  $\gamma(t) = \gamma_0(t)$ if $t \in   T $ and $\gamma(t) = \gamma^i(t)$ if $t \in (a_i,b_i)$  for some $i\in\N$   that is uniquely determined by $t$.   Then, by the length  estimates \eqref{e.gamma0_Lestimate} and \eqref{eq:lenestgammai},
  followed by inequality \eqref{e.d_i_estimates},   we obtain that
 \begin{align*}
 \len(\gamma) &\leq \len(\gamma_0) + \sum_{i\in\N} \len(\gamma^i) \\&   \leq     \nu   d(x,\Omega^c) +  \nu  \sum_{i\in\N}      d_i  \leq   2 \nu  d(x,\Omega^c) \leq N d(x,\Omega^c)\,.
 \end{align*}
  From this it follows that    $\gamma \in   \Gamma(X)^N_{x,\Omega^c}$;  we remark that the required continuity and connecting
  properties of $\gamma$ are straightforward establish,
  and  we omit the details.   Also,
  by 
inequalities \eqref{e.gamma0_Lestimate}, \eqref{e.K_estimate}, \eqref{e.d_i_estimates} and \eqref{eq:intestgammai},  we have
  \begin{align*}
      \int_\gamma g \,ds &
         =  \int_{T} g(\gamma_0(t)) \,dt  + 
      \sum_{i\in \N} \int_{\gamma^i} g \,ds \\
                        &  \leq M^{i_0} \tau    \nu   d(x,\Omega^c) + 3 C_{\mathrm{A}} M^{i_0}  \tau\delta M^{-i_0 q/p}   d(x,\Omega^c) \\
                        & \leq (M^{i_0}   \nu  +  3 C_{\mathrm{A}}  M^{i_0})\tau d(x,\Omega^c)\\&
                        \le S\tau d(x,\Omega^c)\,.
 \end{align*}
In the present case, we have now constructed a  curve $\gamma$ such that
 inequality \eqref{e.desired_curve} holds,  
even without the absorption term. Hence, we  are done   in the first case of \eqref{e.cases}.  
 
  Next we   consider the  slightly more complicated  case $O= (a_0,b_0)\cup  \bigcup_{i \in \N} (a_i,b_i)$, in which there
 is also a final gap $(a_0,b_0)$ such that   $b_0=\len(\gamma_0)$   and $\gamma_0(b_0)\in \Omega^c$. 
 As in the previous case, for   each $i\in \N$, we can   first    construct curves $\gamma^i \co [a_i,b_i] \to X$ such that 
 \begin{equation}\label{eq:lenestgammaiR}
    \len(\gamma^i) \leq   \nu    d(\gamma_0 (a_i), \gamma_0 (b_i)) =   \nu   d_i\,,
 \end{equation}
  and
   \begin{equation}\label{eq:intestgammaiR}
    \int_{\gamma^i} g \,ds \leq 3  C_{\mathrm{A}} d( \gamma_0(a_i), \gamma_0(b_i)) M^{i_0} \tau = 3 C_{\mathrm{A}} M^{i_0} \tau d_i\,.
 \end{equation}
  For   $i=0$ we have to be more careful,
since $\gamma_0(b_0)\not\in\Omega\setminus E_{i_0}$.
  We now proceed as follows.   By using \eqref{e.d_i_estimates} and the 
  equality $K\delta= \kappa$, we first observe that
\[
Kd(\gamma_0(a_0),\Omega^c)\le Kd(\gamma_0(a_0),\gamma_0(b_0))
=Kd_0\le  K \delta d(x,\Omega^c)\le \kappa d(x,\Omega^c)\,.
\]
On the other hand, we still have that $\gamma_0(a_0) \in   \Omega\setminus E_{i_0} $,  and thus
\[
\M_{q,Kd(\gamma_0(a_0),\Omega^c)} g(\gamma_0(a_0)) \le \M_{q,\kappa d(x,\Omega^c)} g(\gamma_0(a_0)) \leq M^{i_0} \tau\,.
\] 
From this it follows that $g\in \mathcal{E}^{K,M^{i_0} \tau}_{q,\gamma_0(a_0),\Omega}$.
 By definition \eqref{eq:alphaDef}  of the function $\alpha_{q,\Omega}(N,K,M^{i_0} \tau)$,   we obtain  a curve   $\gamma^0\co [a_0,b_0]\to X$   connecting $\gamma_0(a_0) \in \Omega$ to $\Omega^c$ such that
  \begin{equation}\label{eq:lenestgammaiInit}
\len(\gamma^0) \leq   Nd(\gamma_0(a_0),\Omega^c) \le N  d(\gamma(a_0),\gamma(b_0)) = N d_0
 \end{equation}
and  
\begin{equation}\label{eq:intestgammaiR_0}
\begin{split}
\int_{\gamma^0} g\,ds &\leq d(\gamma_0(a_0),\Omega^c)\alpha_{q,\Omega}(N,K,M^{i_0}\tau) +   \underbrace{\tau d(x,\Omega^c)}_{>0}  \\
&\le  d_0 \alpha_{q,\Omega}(N,K,M^{i_0} \tau)+\tau d(x,\Omega^c)\,.
    \end{split}
 \end{equation}
Now we define   $\gamma$   as in the first case  but using also the 
final gap $(a_0,b_0)$ by setting $\gamma(t)=  \gamma^0(t) $ for 
  every   $t \in (a_0,b_0  ] $. Then by 
\eqref{e.gamma0_Lestimate},  \eqref{e.d_i_estimates}, \eqref{eq:lenestgammaiInit}, and 
our choice of $N$  and $\delta$,  we obtain
 \begin{align*}
 \len(  \gamma ) &\leq \len(\gamma_0) + \len(\gamma^0) + \sum_{i\in \N} \len(\gamma^i) \\
 &\leq (  \nu  +   \delta N  + \nu) d(x,\Omega^c) \leq N d(x,\Omega^c)\,.
 \end{align*}
 Thus, we find that $  \gamma  \in \Gamma(X)^N_{x,\Omega^c}$. Finally, by inequalities  \eqref{e.gamma0_Lestimate}, \eqref{e.K_estimate}, \eqref{e.d_i_estimates}, \eqref{eq:intestgammaiR}, and \eqref{eq:intestgammaiR_0} we have 
\begin{align*}
\int_{ \gamma } g\,ds    &=  \int_{T} g(\gamma_0(t))\,dt   + 
\sum_{i\in\N} \int_{\gamma^i} g \,ds + \int_{\gamma^0} g \,ds  \\
                    & \leq  M^{i_0} \tau    \nu   d(x,\Omega^c) + 
                     3 C_{\mathrm{A}} M^{i_0}  \tau   d(x,\Omega^c) + d_0 \alpha_{q,\Omega}(N,K,M^{i_0} \tau)+  \tau d(x,\Omega^c)  \\
                        & \leq  S\tau d(x,\Omega^c) + \delta M^{-i_0 q/p} \alpha_{q,\Omega}(N,K,M^{i_0} \tau) d(x,\Omega^c)\,.
 \end{align*}
Recall that $i_0\in \{1,\ldots,k\}$.  
Hence, the desired estimate \eqref{e.desired_curve} for $\gamma$ follows
and thus the proof is complete.
\end{proof}

\section{Main results}\label{s.main}

As a consequence of Theorem \ref{t.converse_and_impro}
and Lemma \ref{lem:alphaClass}, we obtain the following theorem. 
It is the main result of the present paper.

\begin{theorem}\label{c.main}
Let $1\le  p'<p<\infty$. Assume that $X$ supports
a $p'$-Poincar\'e inequality \eqref{e.poincare}.
Let $\emptyset\not=\Omega\subsetneq X$ be an open set
that satisfies a pointwise $p$-Hardy inequality \eqref{e.pointwise}.
Then there exists an exponent   $q\in (p',p)$ such that  
$\Omega$ satisfies a pointwise $q$-Hardy inequality. 
\end{theorem}

\begin{remark}\label{r.quantitative}
The conclusion of Theorem \ref{c.main} reads as follows: there exists $q\in (p',p)$
such that  $\Omega$ satisfies a pointwise $q$-Hardy
inequality. We can establish a more quantitative result. Indeed, by examining the proof of Theorem \ref{t.converse_and_impro},
we see that it runs through 
if  $p$, $p'$ and $q$ satisfy the following inequalities 
\[
\max\{1,p/2\}\le p'<q<p\qquad \text { and }\qquad\delta M^{k\frac{p-q}{p}}<1\,,\] where $M=4$,
 $\delta=\frac{1}{4}$ and $\N\ni k>(2^{p}\delta^{-p}C_{\Gamma}^p  D^5 )^{\frac{1}{p-1}}$.
Here $C_\Gamma>0$ is the constant
appearing in inequality \eqref{eq:HardyChar}. This inequality characterizes the pointwise $p$-Hardy inequality.
   Thus, we can choose
\begin{align*}
k\defeq \lceil (8C_\Gamma)^{\frac{p}{p-1}}D^{\frac{5}{p-1}} +1\rceil > (8C_\Gamma)^{\frac{p}{p-1}}D^{\frac{5}{p-1}} =
(2^{p} \delta^{-p} C_{\Gamma}^p   D^5 )^{\frac{1}{p-1}}\,.
\end{align*}
Then
$\delta M^{k\frac{p-q}{p}}<1
\Leftrightarrow 4^{k\frac{p-q}{p}}<4
\Leftrightarrow p-q<\frac{p}{k}$. On the other hand, by examining the proof of
Lemma \ref{lem:HardyChar}, we 
have $C_\Gamma=4C_{\mathrm{H}}$, where
$C_{\mathrm{H}}>0$ is the constant in the assumed
pointwise $p$-Hardy inequality \eqref{e.pointwise}.
All in all,  we find that
if the assumptions
of Theorem \ref{c.main} hold,
\[
\max\{1,p/2\}\le p'<q<p\qquad\text{ and }\qquad p-q<\frac{p}{ \lceil(32C_{\mathrm{H}})^{\frac{p}{p-1}}D^{\frac{5}{p-1}} +1\rceil}\,,
\]
then  $\Omega$ satisfies a pointwise $q$-Hardy inequality.
Rather similar quantitative bounds
 for the self-improvement of $p$-Poincar\'e inequalities can be found in \cite{SEB}.
\end{remark}

\begin{theorem}\label{t.consequences}
Let $1\le  p'<p<\infty$. Assume that $X$ supports
a $p'$-Poincar\'e inequality. 
Let $\emptyset\not=\Omega\subsetneq X$ be an open set.
Then the following conditions are equivalent:
\begin{itemize}
\item[(A)] The open set $\Omega$ satisfies
a pointwise $p$-Hardy inequality;
\item[(B)] There
are constants $\nu>C_{\mathrm{QC}}$, $\kappa\ge 1$  and $C_\alpha >0$ such that,
for any $\tau\ge 0$, we have
\[
\alpha_{p,\Omega}(\nu,\kappa,\tau)\le C_\alpha \tau\,.
\]
\item[(C)] There are constants   $C_\Gamma>0$, $\nu>C_{\mathrm{QC}}$ and $\kappa\ge 1 $ such that for  each  non-negative and bounded $g \in  LC_0(\Omega)$ and every  $x \in \Omega$, we have
\[
\inf_{\gamma \in \Gamma(X)^\nu_{x,\Omega^c}} \int_{\gamma} g \,ds \leq C_\Gamma \,d(x,\Omega^c) \left( \M_{p,\kappa d(x,\Omega^c)} g (x) \right).
\]
\end{itemize}
\end{theorem}

\begin{proof}
The implication from (A) to (B) follows from Theorem \ref{t.converse_and_impro} and the pointwise estimate
$\alpha_{p,\Omega}\le \alpha_{q,\Omega}$ that trivially is valid if $p\ge q$. The converse follows from Theorem \ref{lem:alphaClass}.
The implication from (A) to (C) is a consequence of Lemma \ref{lem:HardyChar}. On the other hand, by inspecting the proof of Theorem \ref{t.converse_and_impro}, we find that
condition (C) implies (A). In particular, the test function
$h$ that is constructed in the proof actually belongs to $LC_0(\Omega)$.
\end{proof}

\begin{remark}
By combining Theorem \ref{c.main} and Theorem \ref{t.consequences} one obtains self-improvement of further inequalities (B) and (C) in Theorem \ref{t.consequences}; these inequalities  are both 
 equivalent with the pointwise $p$-Hardy inequality. We remark that inequality
(C) differs from the characterizing condition appearing in  Lemma~\ref{lem:HardyChar} in that
the test functions $g$ in (C) are required to vanish outside $\Omega$.
The self-improvement results for the conditions (B) and (C) are naturally also subject to a better 
$p'$-Poincar\'e inequality; we omit the explicit formulations.
\end{remark}

\bibliographystyle{abbrv}
\def\cprime{$'$}

\end{document}